\documentclass[12pt]{amsart}
\usepackage{amsmath,amssymb,amsfonts,amsthm,amstext,verbatim}
\RequirePackage[colorlinks,citecolor=blue,urlcolor=blue]{hyperref}
\usepackage{url}
\usepackage{graphicx,psfrag}
\usepackage{accents}

\newcommand{\arx}[1]{\href{http://arxiv.org/abs/#1}{\nolinkurl{arxiv:#1}}}

\theoremstyle{plain}

\newtheorem{theorem}{Theorem}

\newtheorem{lemma}[theorem]{Lemma}

\numberwithin{equation}{section}

\newcounter{mycount}
\newenvironment{romlist}{\begin{list}{\rm(\roman{mycount})}%
   {\usecounter{mycount}\labelwidth=1cm\itemsep 0pt}}{\end{list}}

\newcommand\s{\sigma}

\newcommand\oo{\infty}

\newcommand\ZZ{{\mathbb Z}}

\newcommand\la{\lambda}

\newcommand\si{\sigma}
\newcommand\eps{\epsilon}

\newcommand\g{\gamma}
\newcommand\resp{respectively}
\newcommand\de{\delta}

\newcommand\Aut{\text{\rm Aut}}

\newcommand\br{{\rm{br}}}
\newcommand\gr{{\rm{gr}}}
\newcommand\pc{p_{\rm c}}
\newcommand\muF{\mu^{{\rm F}}}
\newcommand\muB{\mu^{{\rm B}}}
\newcommand\muFB{\mu^{{\rm FB}}}
\newcommand\siF{\s^{{\rm F}}}
\newcommand\siB{\s^{{\rm B}}}
\newcommand\siFB{\s^{{\rm FB}}}

\newcommand\qq{\qquad}

\newcommand\df[1]{{\bf#1}}
\newcommand\hTF{\widehat T^{\rm F}}

\newcommand\ovG{\accentset{\leftarrow}{G}}
\newcommand\oovG{\accentset{\leftarrow}{\ovG}}

\newcommand{\comm}[1]{}

\newcommand\TF{T^{{\rm F}}}

\DeclareMathOperator{\stab}{Stab}

\begin{document}
\title
{Extendable self-avoiding walks}
\author[Grimmett]{Geoffrey R.\ Grimmett}
\address{(GRG) Statistical Laboratory, Centre for
Mathematical Sciences, Cambridge University, Wilberforce Road,
Cambridge CB3 0WB, UK}
\email{\href{mailto:grg@statslab.cam.ac.uk}{g.r.grimmett@statslab.cam.ac.uk}}
\urladdr{\url{http://www.statslab.cam.ac.uk/~grg/}}

\author[Holroyd]{Alexander E.\ Holroyd}
\address{(AEH) Microsoft Research, 1 Microsoft Way, Redmond, WA 98052, USA}
\email{\href{mailto:holroyd@microsoft.com}{holroyd at microsoft.com}}
\urladdr{\url{http://research.microsoft.com/~holroyd/}}

\author[Peres]{Yuval Peres}
\address{(YP) Microsoft Research, 1 Microsoft Way, Redmond, WA 98052, USA}
\email{\href{mailto:peres@microsoft.com}{peres@microsoft.com}}
\urladdr{\url{http://research.microsoft.com/~peres/}}

\begin{abstract}
The connective constant $\mu$ of a graph is the exponential growth rate of
the number of $n$-step self-avoiding walks starting at a given vertex.  A
self-avoiding walk is said to be \emph{forward} (\resp, \emph{backward})
\emph{extendable} if it may be extended forwards (\resp, backwards) to a
singly infinite self-avoiding walk. It is called \emph{doubly extendable} if
it may be extended in both directions simultaneously to a doubly infinite
self-avoiding walk. We prove that the connective constants for forward,
backward, and doubly extendable self-avoiding walks, denoted \resp\ by
$\muF$, $\muB$, $\muFB$, exist and satisfy $\mu=\muF=\muB=\muFB$ for every
infinite, locally finite, strongly connected, quasi-transitive directed
graph. The proofs rely on a 1967 result of Furstenberg on dimension, and
involve two different arguments depending on whether or not the graph is
unimodular.
\end{abstract}

\date{25 July 2013 (revised 12 November 2013)}  

\keywords{self-avoiding walk, connective constant, transitive graph,
quasi-transitive graph, unimodular graph, growth, branching number}
\subjclass[2010]{05C30, 82B20, 60K35}

\maketitle

\section{Introduction}\label{sec:intro}

Let $G=(V,E)$ be an infinite, strongly connected, locally finite, directed
graph (possibly with parallel edges), and let $\s_n(v)$ be the number of
$n$-step self-avoiding walks (SAWs) on $G$ starting at the vertex $v \in V$
and directed away from $v$.  Hammersley proved in 1957 \cite{jmhII} that the
limit
\begin{equation}\label{mudef}
\mu := \lim_{n \to\oo} \Bigl(\sup_{v \in V}  \s_n(v)\Bigr)^{1/n}
\end{equation}
exists, and that if $G$ is quasi-transitive then
\begin{equation}\label{connconst}
\lim_{n \to \oo} \s_n(v)^{1/n} = \mu \qquad \text{for all }v \in V.
\end{equation}

The constant $\mu=\mu(G)$ is called the \df{connective constant} of $G$. Note
that \eqref{mudef} is not necessarily the natural definition of connective
constant for a general graph, see \cite{GrL1,Lac}.  There is no loss of
generality in restricting attention to directed graphs, since each edge of an
undirected graph may be interpreted as a pair of edges with opposite
orientations.

The purpose of this article is to study the growth
rates of the numbers of $n$-step SAWs from $v$ that are extendable
to infinite SAWs at one or both of their ends.

Let $w$ be an $n$-step directed SAW starting at a vertex $v$.
We call $w$ \df{forward extendable} if it is an initial segment of
some singly infinite directed SAW from $v$.
We call $w$ \df{doubly extendable} if it is a sub-walk of some doubly
infinite directed SAW passing through $v$.
We call $w$ \df{backward extendable} if it is the final segment of
some singly infinite directed SAW from infinity, passing through $v$
and ending at the other endpoint of $w$.
Let $\siF_n(v)$, $\siB_n(v)$, and $\siFB_n(v)$ denote the numbers of forward,
backward, and
doubly extendable $n$-step SAWs from $v$, respectively.
We define $\muF$, $\muB$, and $\muFB$ analogously to \eqref{mudef} whenever the limits exist.

\begin{theorem}\label{main}
Let $G$ be an infinite, locally finite, strongly connected, quasi-transitive directed graph.
\begin{romlist}
\item The limits $\muF$, $\muB$, $\muFB$ exist and satisfy $\mu=\muF=\muB=\muFB$.
\item We have
\begin{equation}\label{connconst2}
\lim_{n \to \oo} \siF_n(v)^{1/n} = \muF \qq \text{for all }v \in V.
\end{equation}
\end{romlist}
\end{theorem}

The analogue of \eqref{connconst2} does not hold in general for backward or
doubly extendable walks.  For example, if $v$ has only one neighbour (joined
to it by edges in both directions) then $\siB_n(v)$ and $\siFB_n(v)$ are both
$0$ for $n\geq 1$.

A principal ingredient of the proof of Theorem~\ref{main} is
a result of Furstenberg \cite[Prop.\ III.1]{Fur67} from 1967;
a recent exposition appears in \cite[Sect.\ 3.3]{LyP}.
The same method provides an alternative
proof of Hammersley's result \eqref{connconst}.

Theorem~\ref{main}(i) states that the exponential growth rates of counts of
SAWs coincide for the four types of SAW under consideration. One may ask also
about more refined asymptotic properties. Suppose $G$ satisfies the
conditions of Theorem~\ref{main}, and is for simplicity vertex-transitive.
The sub-multiplicativity of SAW-counts (see the proof of Lemma~\ref{half}
below) gives that $\sigma_n^\bullet \ge (\mu^\bullet)^n$ for each of
$\bullet\in\{\;\;, \mathrm{F}, \mathrm{B}, \mathrm{FB}\}$.  Therefore,
whenever it is known that $\si_n \le A \mu^n$ for some $A<\oo$, it follows
that $A^{-1} \le \si^\bullet_n/\si_n \le 1$. This is indeed the situation for
the (undirected) integer lattice $\ZZ^d$ with $d \ge 5$, by
\cite[Thm~1.1(a)]{HSa}.  We do not know whether $\sigma_n^\bullet$ and
$\sigma_n$ agree up to a multiplicative constant for every $G$ satisfying the
conditions of Theorem~\ref{main}.  The square lattice $\ZZ^2$ is a
particularly interesting case.

In the case $G = \ZZ^d$ with $d \ge 2$, the method of `bridges' developed by
Hammersley and Welsh \cite{HW62} immediately gives the results of
Theorem~\ref{main}, and furthermore shows that $\si^\bullet_n/\si_n \geq
\exp(-c\sqrt{n})$ for some $c=c(d)>0$.  An interesting related
notion of `endless SAWs' is studied in \cite{clisby}.

The proof of Theorem~\ref{main} is divided into several parts. The proofs of
the equalities $\mu=\muF$ and $\muB=\muFB$ of part (i) use a result of
Furstenberg, namely that a subperiodic tree has growth rate equal to its
branching number.  This is applied to certain trees constructed from the sets
of SAWs (of the various types) from a given vertex, and it is argued that the
branching numbers of the appropriate pairs of trees coincide.  See
Section~\ref{sec:trees}.  A related argument gives part (ii).  The remaining
equality is proved by two different arguments depending on whether or not $G$
is unimodular.  In the unimodular case, a mass-transport argument yields
$\muF=\muB$ (Section~\ref{sec:unimodular}), while in the non-unimodular case
(Section~\ref{sec:geo}) we show the existence of a `quasi-geodesic' of a
certain type, and employ a counting argument related to Hammersley's methods
of \cite{jmhII} to obtain $\mu=\muB$. In Section~\ref{sec:def} below we
define the various concepts referred to above. Many of our arguments can be
simplified if $G$ is undirected and/or transitive, and we indicate such
simplifications where appropriate.

\section{Preliminaries}\label{sec:def}

In this section we introduce terminology and our main tools. Let $G=(V,E)$ be
a directed graph, possibly with parallel edges. We
call $G$ \df{locally finite} if each vertex has finite in-degree and
out-degree.

An \df{automorphism} is
a bijection $\g:V\to V$ such that, for all $v,w\in V$, the number of edges that are directed from $v$ to
$w$ equals the number that are directed from $\g(v)$ to $\g(w)$. The automorphisms of
$G$ form the \df{automorphism group} $\Aut(G)$.
The orbits of $V$ under $\Aut(G)$ are called \df{transitivity
classes}, and $G$ is \df{transitive} if it has only one transitivity class,
or \df{quasi-transitive} if it has finitely many transitivity classes.

A \df{walk} $w$ consists of a sequence of vertices $(v_i)_{m< i< n}$ together
with edges $(e_i)_{m<i<i+1<n}$, where $e_i$ is a directed edge from $v_i$ to
$v_{i+1}$, and where $-\infty\leq m\leq n\leq\infty$.  The length $|w|$ of
$w$ is the number of its edges. The walk is \df{singly infinite} if either
$m\in\ZZ$ and $n=\oo$ or $m=-\oo$ and $n\in\ZZ$, and \df{doubly infinite} if
$m=-\oo$ and $n=\oo$. A graph $G$ is \df{strongly connected} if for every
pair $u,v\in V$ there exist finite walks from $u$ to $v$ and from $v$ to $u$.

A \df{self-avoiding walk} (SAW) on $G$ is a walk all of
whose vertices are distinct.
SAWs may be finite, singly infinite, or doubly infinite.
Let $\s_n(v)=\s_n(v,G)$ be the number of length-$n$ SAWs
 starting at $v\in V$.
In the presence of parallel edges, two SAWs with identical vertex-sequences
but different edge-sequences are considered distinct. We write
$\s_n=\s_n(G):=\sup_{v\in V} \s_n(v)$, and denote
$$
\mu=\mu(G):=\lim_{n\to\infty} \s_n^{1/n},
$$
whenever the limit exists.
Forward, backward, and doubly extendable SAWs are defined as in the
introduction, and the quantities $\siF_n(v)$, $\siF_n$, 
$\muF$, etc., are defined analogously.

We turn now to certain elements in the study of trees, for which we follow
\cite[Chap.~3]{LyP}. Let $T=(W,F)$ be an infinite, locally finite, 
undirected tree with
root $o$. For $v \in W$, the distance between $v$ and $o$ is written $|v|$.
For $e=\langle v_1,v_2\rangle\in F$, let $|e| = \max\{|v_1|, |v_2|\}$. Let
$W_n =\{v \in W: |v| = n\}$ be the set of vertices at \df{level} $n$. A
\df{cutset} is a minimal set of edges whose removal leaves $o$ in a finite
component. Since $T$ is assumed locally finite, cutsets are finite.

There are two natural notions of dimension of a tree $T$.  The \df{growth} is
given by
$$
\gr(T) := \lim_{n\to\oo} |W_n|^{1/n},
$$
whenever this limit exists.
In any case, the lower growth and upper growth are given respectively by
$$
\underline{\gr}(T) := \liminf_{n\to\oo} |W_n|^{1/n},\qq
\overline{\gr}(T) := \limsup_{n\to\oo} |W_n|^{1/n}.
$$
A more refined notion is the \df{branching number}
\begin{equation}\label{def:br}
\br(T) := \sup\biggl\{\la: \inf_\Pi \sum_{e \in \Pi} \la^{-|e|}>0 \biggr\},
\end{equation}
where the infimum is over all cutsets $\Pi$ of $T$.
One interesting property, which is sometimes helpful for intuition,
is that the critical probability $\pc(T)$ of bond percolation on $T$ satisfies
\begin{equation}\label{perc}
\pc(T)=1/\br(T).
\end{equation}
See \cite[Thm~6.2]{Lyons90} or \cite[Thm~5.15]{LyP}.

The growth and branching number of a general tree need not be equal (and indeed the
growth need not exist).  However, we have the following inequality \cite[eqn (1.1)]{LyP}.
We include a proof for the reader's convenience.

\begin{lemma}\label{lem:1}
For any locally finite, infinite rooted tree, $\br(T) \le \underline\gr(T)$.
\end{lemma}

\enlargethispage*{1cm}
\begin{proof}
Let $\la>\underline\gr(T)$.
Taking $\Pi$ as the set of edges joining $W_{n-1}$ and $W_n$, we have that
$$
\inf_\Pi\sum_{e \in \Pi} \la^{-|e|}\le |W_n| \la^{-n}.
$$
There exists a subsequence $(n_i)$ along which the last term tends to zero.
\end{proof}

Furstenberg \cite{Fur67} gave a condition under which branching number and
growth do coincide. For $w \in W$, denote by $T^w$ the sub-tree of $T$
comprising $w$ and its descendants, considered as a rooted tree with root
$w$. Let $N \ge 0$. The tree $T$ is called \df{$N$-subperiodic} if, for all
$w \in W$, there exists $w'$ with $|w'|\leq N$ such that there is an
injective graph homomorphism from $T^w$ to $T^{w'}$ mapping $w$ to $w'$. If
$T$ is $N$-subperiodic for some $N$, we call $T$ \df{subperiodic}.

\begin{theorem}[Furstenberg \cite{Fur67}]\label{furst}
Let $T$ be an infinite, locally finite, root\-ed tree.  If $T$ is subperiodic then
$\gr(T)$ exists and equals $\br(T)$.
\end{theorem}

For a proof see \cite[Prop.\ III.1]{Fur67} or \cite[Sec.~3.3]{LyP}.

Finally in this section we introduce the notions of unimodularity and the
mass-transport principle on graphs; more details may be found in \cite{BLPS}
and \cite[Chap.~8]{LyP}. Let $G$ be an infinite, locally finite, strongly
connected, quasi-transitive directed graph.  The stabiliser $\stab(u)$ of a
vertex $u$ is the set of automorphisms that preserve $u$, and $\stab(u)v$
denotes the orbit of a vertex $v$ under this set.  We may define a positive
\df{weight} function $M:V\to(0,\infty)$  via
\begin{equation}\label{weight}
\frac{M(u)}{M(v)}=\frac{|\stab(u) v|}{|\stab(v) u|},\qq u,v \in V,
\end{equation}
where $|\cdot|$ denotes cardinality.
The function $M$ is uniquely defined up to
multiplication by a constant, and is
automorphism-invariant up to multiplication
by a constant.
The graph $G$ is called \df{unimodular} if $M$ is constant on each transitivity class.
The following fact is very useful.

\begin{theorem}[Mass-transport principle]\label{mt}
Let $G=(V,E)$ be an infinite, locally finite, strongly connected,
quasi-transitive directed graph with weight function $M$. Suppose that $G$ is
unimodular, and let $S$ be a set comprising a representative from each
transitivity class of $G$.  If $m:V\times V\to [0,\infty]$ satisfies $m(\phi
u,\phi v)=m(u,v)$ for all $u,v\in V$ and every automorphism $\phi$ of $G$,
then
$$
\sum_{\substack{s\in S,\\v\in V}} M(s)^{-1} m(s,v)=
\sum_{\substack{s\in S,\\v\in V}} M(s)^{-1} m(v,s).
$$
\end{theorem}
For proofs of Theorem~\ref{mt} and the immediately preceding assertions see
e.g.\ \cite[Thm~8.10, Cor.~8.11]{LyP}.
To obtain the above formulation, the
results of \cite{LyP} are applied to the undirected graph $G'$ obtained
from $G$ by
ignoring edge orientations, with the automorphism group of the directed graph $G$.  Thus the assumption (in Theorem \ref{mt})
of strong connectivity
of $G$ may be weakened to that of connectivity of $G'$, in which case
we say that $G$ is \df{weakly connected}.

\section{SAW trees}\label{sec:trees}

The proof of Theorem~\ref{main} is divided into several parts.
We start by establishing two of the inequalities required for part (i).

\begin{lemma}\label{half}
Under the assumptions of Theorem~\ref{main}, the limits $\muF$, $\muB$, $\muFB$
exist and satisfy $\mu=\muF$ and $\muB=\muFB$.
\end{lemma}

\begin{proof}
It is standard that $\si_n$ satisfies the submultiplicative inequality
$$
\s_{m+n} \le \s_m \s_n,
$$
and it is easy to see that the sequences $\siF_n$, $\siB_n$, $\siFB_n$
satisfy the same inequality. The existence of the constants $\muF$, $\muB$,
$\muFB$ follows by the subadditive limit theorem (see, for example,
\cite[Ex.~3.9]{LyP}).

Fix $v\in V$. From $G$ we construct the rooted \df{SAW tree} $T(v)$ as
follows. The vertices of $T(v)$ are the finite SAWs from $v$, with the
trivial walk of length $0$ being the root. Two vertices of $T(v)$ are
declared adjacent if one walk is an extension of the other by exactly one
step.

Let $S$ be a set of vertices comprising one representative of each
transitivity class of $G$, and let
$$
T:=\bigvee_{s\in S} T(s)
$$
be the rooted tree obtained from disjoint copies of the trees $T(s)$ by
joining their roots to one additional vertex $o$, which is designated the
root of the resulting tree. (In the case of transitive $G$, the argument may
be simplified by instead taking $T=T(v)$ for any fixed $v$.) The level set
$W_n$ of $T$ has size
$$
|W_n|=\sum_{s\in S} \s_{n-1}(s),
$$
and hence,
$$
\mu=\liminf_{n\to\infty} \sigma_{n-1}^{1/n}
\leq \underline{\gr}(T) \leq \overline{\gr}(T) \leq
\limsup_{n\to\infty} (|S|\sigma_n)^{1/n}=\mu
$$
so that $\gr(T)=\mu$.

Define $\TF(v)$ to be the \df{forward SAW tree} constructed in an identical
manner to $T(v)$ from all forward-extendable SAWs on $G$ from $v$. Observe
that $\TF(v)$ is precisely the tree obtained from $T(v)$ by removing all
finite bushes, i.e., removing (together with its incident edges) each vertex
$w$ whose rooted subtree $T(v)^w$ is finite. We similarly define
$\TF=\bigvee_{s\in S} \TF(s)$, and note that by the above argument,
$\gr(\TF)=\muF$.

Since both $T$ and $\TF$ are $1$-subperiodic, Theorem~\ref{furst} applies to give
$$
\br(T)=\gr(T),\qq \br(\TF)=\gr(\TF).
$$
On the other hand, since branching number is unaffected by the removal of finite bushes
(by the definition of branching number or by \eqref{perc}), we have $\br(T)=\br(\TF)$.
Thus $\mu=\muF$.

An identical argument gives the equality $\muB=\muFB$: removing all finite bushes
from the \df{backward SAW tree} gives precisely the \df{doubly extendable SAW tree}
(where these objects are defined by obvious analogy with the previous cases).
Thus the two trees have equal branching numbers, whence by Theorem~\ref{furst}
 they have equal growths.
\end{proof}

Our next proof employs similar methods.

\begin{proof}[Proof of Theorem~\ref{main}(ii)]
As in the proof of Lemma~\ref{half} above,
let $\TF:=\bigvee_{s \in S} \TF(s)$,
where $\TF(v)$ is the forward SAW tree from $v$ and $S$ is a set of
representatives of the transitivity classes of $G$.
As argued in the previous proof we have $\br(\TF)=\gr(\TF)=\muF$.
By the definition of branching number (or \eqref{perc}),
\begin{equation}\label{4.1}
\br(\TF) = \max_{s \in S} \br(\TF(s)).
\end{equation}
Therefore there exists $t \in S$ such that $\br(\TF(t))= \muF$.

For every vertex $v$ of $G$ we have $\br(\TF(v))\leq \muF$.  Call $v$
\df{good} if equality holds, or \df{bad} if the inequality is strict.
We showed above that good vertices exist.  We will see that in fact there are no bad vertices.

For any good vertex $u$, construct the `pruned' tree
$\hTF(u)$ from $\TF(u)$ by removing the subtree
$\TF(u)^w$ rooted at each vertex $w$ of $\TF(u)$ that corresponds to a bad vertex of $G$
(i.e., that represents a walk from $u$ ending at a bad vertex).
Since each removed
subtree has branching number less than $\muF-\epsilon$ for some fixed
$\epsilon>0$ depending only on the graph, we have
(by the definition of branching number or \eqref{perc}) that
$$
\br(\hTF(u))=\br(\TF(u))=\muF.
$$
By Lemma \ref{lem:1},
\begin{equation}\label{4.2}
\underline\gr(\hTF(u)) \ge \br(\hTF(u))=\muF.
\end{equation}

Let $\widehat G$ be the subgraph of $G$ induced by the set of all good
vertices, and observe that $\hTF(u)$ is precisely the forward SAW tree from
$u$ on $\widehat G$.  Thus \eqref{4.2} gives that for any good
$u$,
\begin{equation}\label{4.3}
\liminf_{n\to\infty} \siF_n(u,\widehat G)^{1/n}\geq \muF.
\end{equation}

\enlargethispage*{1.5cm}
Finally, from any vertex $v$ of $G$ there exists a
finite directed walk to some good vertex. Take such a walk of minimum length,
say length $d$ and ending at $u$.  Then
$$
\siF_n(v,G)\geq \siF_{n-d}(u, \widehat G),
$$
so by \eqref{4.3} we have $\liminf \siF_n(v)^{1/n}\geq\muF$.  Since by
definition of $\muF$ we have $\limsup \siF_n(v)^{1/n}\leq\muF$, the result follows.
\end{proof}

\begin{proof}[Alternative proof of Hammersley's result \eqref{connconst}]
The above proof goes through with each $\TF(s)$ replaced by the ordinary SAW tree $T(s)$,
and with $\muF$ replaced by $\mu$.
\end{proof}

\section{The unimodular case}
\label{sec:unimodular}

The proof of the remaining equality of Theorem~\ref{main}(i) is further
divided into two cases according to whether or not the graph is unimodular.
In the former case, a stronger statement holds.  Let $\ovG$ be the directed
graph obtained by reversing all edges of $G$.  Recall that
$\siF_n(G):=\sup_{v\in V} \siF_n(v,G)$, and similarly for $\siB_n$.

\begin{lemma}\label{reverse}
Under the assumptions of Theorem~\ref{main}, suppose in
addition that $G$ is unimodular.  There exists $C=C(G)\ge 1$
such that
$$
C^{-1}\leq \frac{\siF_n(G)}{\siB_n(\ovG)}\leq C, \qq n \ge 0.
$$
If in addition $G$ is transitive then we may take $C=1$.
\end{lemma}

\begin{proof}
Let $S$ be a set of representatives of the transitivity classes of $G$, and
let $M$ be the weight function as in \eqref{weight}.  There exists $c\geq1$
such that $c^{-1}\leq M(s)/M(s')\leq c$ for all $s,s'\in S$, with $c=1$ in
the transitive case. By Theorem~\ref{mt} with $m(u,v)$ defined to be the
number of length-$n$ forward-extendable walks from $u$ to $v$ on $G$,
$$\sum_{s\in S} M(s)^{-1} \siF_n(s,G)=\sum_{s\in S} M(s)^{-1}
\siB_n(s,\ovG).$$ We deduce the claimed inequalities with $C=c|S|$.
\end{proof}

If $G$ is an \emph{undirected} unimodular graph (where as usual we interpret
an undirected edge as a pair of edges with opposite orientations), then
$\ovG$ and $G$ are isomorphic, so Lemma~\ref{reverse} immediately gives
$\muB=\muF$, establishing Theorem~\ref{main} in this case.

A directed graph $G$ need not be isomorphic to $\ovG$: an infinite, transitive,
unimodular counterexample is given in Figure~\ref{snub}.
(A finite counterexample may be obtained by orienting the snub
cube in a similar manner.)  Nonetheless, we obtain a simple proof of
Theorem~\ref{main} in the unimodular case, as follows.
\begin{figure}
\begin{center}
\includegraphics[width=0.6\textwidth]{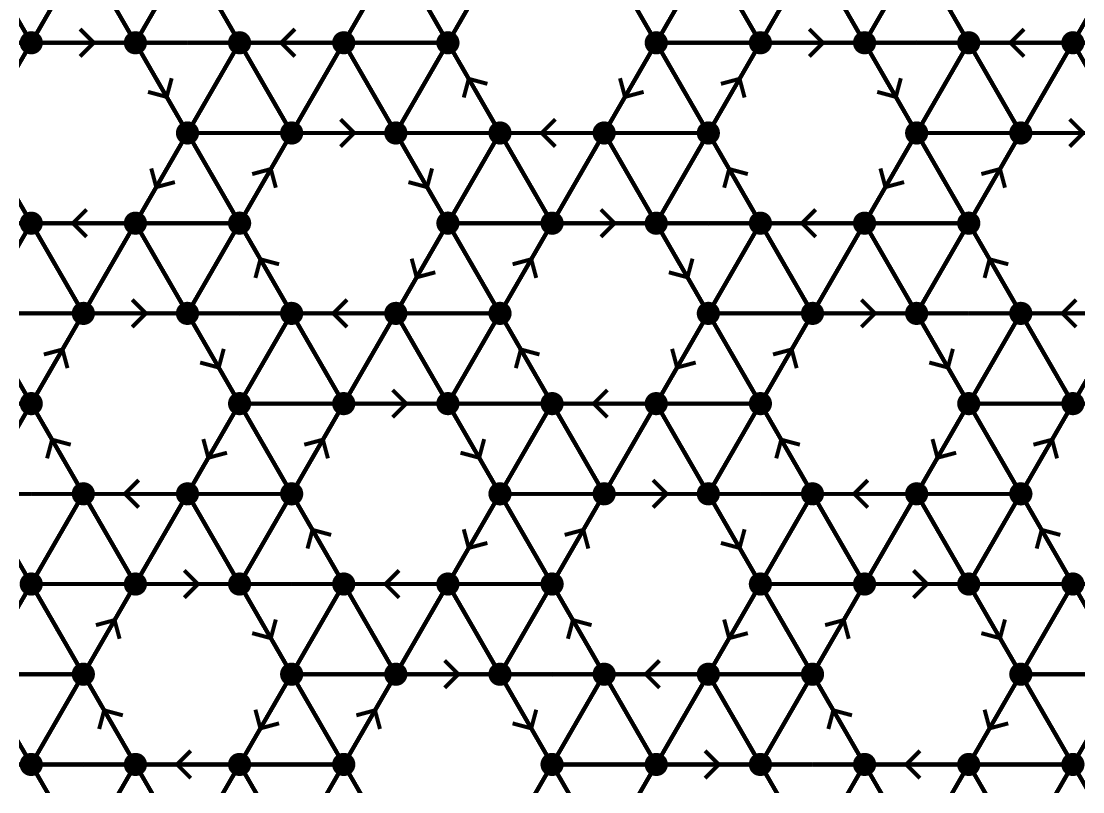}
\end{center}
\caption{An infinite, transitive, directed graph $G$ that is not isomorphic
to its edge-reversal $\protect\ovG$.  An undirected edge is interpreted
as a pair of edges with opposite orientations.}\label{snub}
\end{figure}

\begin{proof}[Proof of Theorem~\ref{main}(i), unimodular case]
Suppose $G$ is a  unimodular graph.  By Lemma \ref{reverse}, $\muF(G)=\muB(\ovG)$, so by
Lemma~\ref{half},
$$
\muB(G)\leq\mu(G)=\muF(G)=\muB(\ovG).
$$
We apply the same argument to $\ovG$, noting that $\oovG=G$, to
obtain $\muB(\ovG) \le \muB(G)$, so that equality holds throughout.
Combined with Lemma~\ref{half}, this concludes the proof.
\end{proof}

\noindent
{\bf Remarks.}
The assumption of strong connectivity is optional
for a unimodular graph $G$. The above proof
of Theorem~\ref{main}(i) is valid if
$G$ is weakly connected in the sense explained after the statement of Theorem~\ref{mt}. If $G$ is not even weakly connected, the same
conclusion holds for each weakly connected component of $G$, and hence for $G$ also.  We remark also that for any transitive $G$ that is weakly connected but not strongly connected, a simple argument shows that {\em all} SAWs are doubly extendable, so that the claims of Theorem 1 hold trivially in this case.

\section{Geodesics}
\label{sec:geo}

By Lemma~\ref{half}, $\mu=\muF$ and $\muB=\muFB$, and all that remains is to
prove the missing equality in the non-unimodular case.  Before doing this we
present a proof in the simpler case when $G$ is undirected. This proof
applies to both unimodular and non-unimodular undirected graphs, but the
subsequent proof for the directed case requires non-unimodularity.

In an undirected graph $G$, a singly infinite walk with vertex sequence $(v_i)_{i\geq 0}$
is called a \df{geodesic} if for all $i,j\geq 0$, the graph-distance between
$v_i$ and $v_j$ is $|i-j|$. By a standard compactness argument
(see, for example, \cite[Thm~3.1]{Wat}), in any infinite, locally finite, connected,
undirected graph there is a geodesic starting from any given vertex.

\begin{proof}[Proof of Theorem~\ref{main}(i) for undirected graphs]
Let $G$ be an un\-directed graph, and let $v \in V$.  Fix a geodesic
$\gamma=(v_i)_{i\geq 0}$ started at $v$.  For a SAW $w$ of length $n$ from
$v$, let $L$ be the largest integer for which $v_L$ lies on $w$.  Let $w^-$
and $w^+$ be the portions of $w$ from $v_L$ to $v$ (reversed), and from $v_L$
to the endpoint of $w$, respectively. See Figure \ref{fig1}. Both $w^-$ and
$w^+$ are backward extendable via the sub-walk $(v_L,v_{L+1},\dots)$ of $\g$.
Since $\g$ is a geodesic, we have that $L\le |w^-|\leq n$. Therefore,
$$
\sigma_n(v) \le \sum_{L=0}^n \sum_{k=L}^n \siB_k(v_L) \siB_{n-k}(v_L),
$$
where $k$ represents $|w^-|$.  By the definition of $\muB$,
for every $\epsilon >0$ there exists $C=C(\epsilon)<\oo$ such that
$\siB_n\leq C(\muB+\epsilon)^n$, and therefore
$$
\si_n(v) \le C^2(n+1)^2 (\muB+\eps)^n, \qq n \ge 1,
$$
so that $\mu \le \muB$.  Clearly $\muB\leq \mu$, so combining with
Lemma~\ref{half} gives the result.
\end{proof}

\begin{figure}\normalsize
\begin{center}
	\psfrag{v}{$v$}
	\psfrag{vL}{$v_L$}
    \psfrag{g}{$\gamma$}
	\psfrag{wm}{$w^-$}
	\psfrag{wp}{$w^+$}
\includegraphics[width=0.7\textwidth]{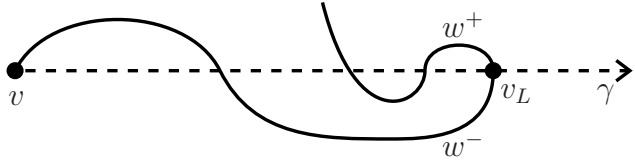}
\end{center}
\caption{The proof of $\mu\leq \muB$ for undirected graphs: both portions
$w^-,w^+$ of the walk $w$ (solid) are backward extendable via the geodesic $\gamma$ (dashed).}\label{fig1}
\end{figure}

Returning to the directed case, we will use the following concept. For
$\alpha>0$, an $\alpha$-\df{quasi-geodesic} of a directed graph $G$ is a
doubly infinite sequence of vertices $(v_i)=(v_i)_{i\in\ZZ}$ satisfying
$$
d_G(v_i,v_j)\geq \alpha|i-j|\qq \text{for all }
i,j\in\ZZ,
$$
where $d_G$ denotes \emph{undirected} graph-distance on $G$
(that is, the length of a shortest path with edges directed arbitrarily). Note that
an $\alpha$-quasi-geodesic is necessarily self-avoiding.
A sequence $(v_i)_{i\in\ZZ}$ is a \df{quasi-geodesic} if it is
an $\alpha$-quasi-geodesic for some $\alpha>0$.

\begin{lemma}\label{q-geo}
Suppose the assumptions of Theorem~\ref{main} hold and in addition $G$ is not unimodular.
There exists a quasi-geodesic $(v_i)_{i\in\ZZ}$ such that there are directed
edges from $v_{i+1}$ to $v_i$ and from $v_{-i-1}$ to $v_{-i}$ for each $i\geq 0$.
\end{lemma}

\begin{proof}[Proof of Lemma \ref{q-geo}]
We prove first that there exists $C\geq 1$ such that the weight function $M$ satisfies
\begin{equation}\label{edge-bound}
C^{-1} \leq \frac{M(u)}{M(v)} \leq C, \qquad \langle u,v\rangle\in E.
\end{equation}
Let $v_1,v_2,\dots,v_k$ be representatives of the orbits of $\Aut(G)$, and
denote by $\{w_{ij}: 1 \le j \le d_i\}$ the final endpoints of directed edges emanating from $v_i$.
Let $\langle u,v \rangle \in E$ be directed from $u$ to $v$.
Since $G$ is quasi-transitive, there exists $\g\in \Aut(G)$ and $1\le i \le k$,
$1\le j \le d_i$ such that $\g(v)=v_i$
and $\g(w)=w_{ij}$.  Equation \eqref{edge-bound} follows for some $C \ge 1$
since $M$ is positive and automorphism-invariant up to
a multiplicative constant.

Since $G$ is non-unimodular, we may find two vertices $u_0,u_1$ in the same
transitivity class with unequal weights.  Assume without loss of generality
that $M(u_0)=1$, and that $c := M(u_1)$ satisfies $c>1$.  Let $\phi$ be an
automorphism mapping $u_0$ to $u_1$, and define $u_i=\phi^i(u_0)$ for
$i\in\ZZ$. Since $M$ is automorphism-invariant up to a multiplicative
constant,
\begin{equation}\label{powers}
M(u_i)=c^i,\qquad i\in\ZZ.
\end{equation}
Let $\xi$ be the vertex-sequence of a shortest directed walk from $u_1$ to
$u_0$ (which exists since $G$ is assumed strongly connected), and let $\zeta$
be the vertex-sequence of a shortest directed walk from $u_{-1}$ to $u_0$.
Let $\overline\xi$ denote the sequence $\xi$ in reverse order. Let
$w=(w_i)_{i\in \ZZ}$ be the doubly infinite sequence of
vertices obtained by concatenating the sequences
$$
\dots,\phi^{-2} \zeta,\phi^{-1} \zeta,\zeta,
\overline\xi,\phi^1\overline\xi,\phi^2\overline\xi,\dots
$$
in this order (indexed so that $w_0=u_0$, and omitting the duplicate vertex
where two concatenated sequences meet). Then $w$ forms a doubly infinite path
with its edges directed towards $w_0$, as required for the claimed
quasi-geodesic, but it is not necessarily self-avoiding. By
\eqref{edge-bound} and \eqref{powers} and the fact that the concatenated
walks are bounded in length, there exist $\beta, \gamma >0$ such that
\begin{equation}\label{affine}
d_G(w_i,w_j)\geq \beta|i-j|-\gamma,\qq i,j\in\ZZ.
\end{equation}

We now erase loops from $w$ until we obtain a self-avoiding sequence.  More
precisely, if there exist $a<b$ with $w_a=w_b$, then choose such $a$, $b$
with $|a|+|b|$ minimal (say), and remove $w_{a+1},\dots,w_{b-1},w_b$ from the
sequence. Iterate this indefinitely. Since initially $w$ visited each vertex
only finitely many times, the sequence $(v_i)_{i\in\ZZ}$ of vertices that are
never removed is a self-avoiding sequence. This sequence may furthermore be
indexed so that it still has edges directed towards $v_0$. (We re-index after
each loop-erasure: if the chosen loop $w_a,\ldots,w_b$ does not contain
$w_0$, we preserve the index of $w_0$; if the loop contains $w_0$, we
re-index so that the old $w_a$ becomes the new $w_0$.) Since loop-erasure
does not increase distances along the walk among the vertices that remain,
\eqref{affine} holds with $(v_i)$ in place of $(w_i)$ and the same $\beta$,
$\gamma$.  Since $(v_i)_{i\in\ZZ}$ is self-avoiding we may now adjust $\beta$
so that this inequality holds with $\gamma=0$.
\end{proof}

\begin{proof}[Proof of Theorem~\ref{main}(i), non-uni\-mod\-ular case]\sloppypar
Let $G$ be non-uni\-mod\-ular. By Lemma~\ref{half}, it suffices to prove
$\mu=\muB$, and obviously we have $\muB\leq\mu$. Fix an
$\alpha$-quasi-geodesic $(v_i)_{i\in\ZZ}$ as in Lemma~\ref{q-geo} for some
$\alpha>0$.  Let $w=(w_0,\dots,w_n)$ be the vertex-sequence of an $n$-step
directed SAW starting at $w_0=v_0$.  We will bound the number of such walks
above in terms of $\muB$ by considering various cases. Let $S_+$ be the set
of intersections of $w$ with $\{v_i:i\geq 0\}$, and $S_-$ the set of
intersections of $w$ with $\{v_i:i\leq 0\}$. Note that, if $v_i \in S_+\cup
S_-$ then $|i| \leq n/\alpha$.

Let $\delta\in(0,\frac12)$.
First suppose that $|S_+|\leq \delta n$.  Decompose the walk $w$ into minimal
segments starting and ending with an element of $S_+$, together with
(possibly) a final segment starting in $S_+$.  For each such segment
$w_a,\dots, w_b$, its truncation $w_a,\dots, w_{b-1}$ is
backward extendable via the SAW $\dots, v_{i+2},v_{i+1},v_i$, where
$v_i=w_a$.  See Figure~\ref{qgeo}(i).

\begin{figure}\normalsize
\begin{center}
	\psfrag{v0}{$v_0$}
\includegraphics[width=0.65\textwidth]{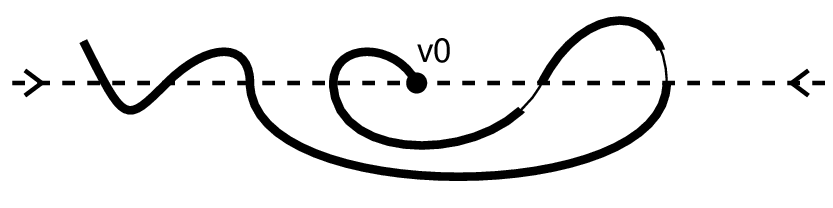}
\\[5mm]
	\psfrag{v0}{$v_0$}
	\psfrag{wa}{$w_a$}
	\psfrag{wb}{$w_b$}
\includegraphics[width=0.65\textwidth]{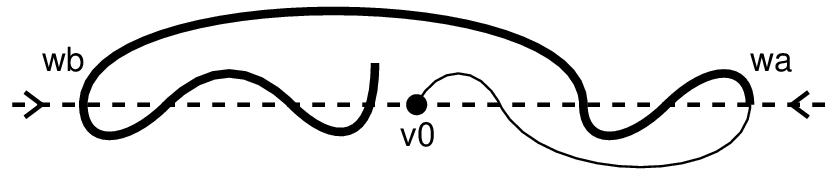}
\end{center}
\caption{Two cases in the proof of $\mu\leq \muB$ for undirected non-unimodular graphs:
(i) the walk $w$ (solid) in the upper figure has few intersections with the right half of the quasi-geodesic (dashed);
(ii) in the lower figure, $w$ has many intersections with both halves of the quasi-geodesic.  In both cases, each
thickened portion of $w$ is backward extendable.}
\label{qgeo}
\end{figure}

If $|S_-|\leq \delta n$ then the walk may similarly be decomposed to give at
most $\delta n$ backward-extendable segments.

Now suppose that $|S_+|,|S_-|> \delta n$.  Let $I_+=\max\{i:v_i\in S_+\}$ and
$I_-=\max\{i:v_{-i}\in S_-\}$, and observe that $I_+,I_->\delta n$.  Thus if
we write $v_{I_+}=w_a$ and $v_{-I_-}=w_b$ then
$$
|a-b|\geq \alpha(I_+  + I_-)>2\alpha\delta n,
$$
since $(v_i)$ is an
$\alpha$-quasi-geodesic.
Writing
$m=\min\{a,b\}$, we deduce that $w_m,w_{m+1},\dots, w_n$ is a
backward-extendable SAW of length greater than $2\alpha\delta n$.
See Figure~\ref{qgeo}(ii).

Combining the various cases, we obtain
\begin{align}\label{bound}
\sigma_n(v_0)  &\leq
2 \sum_{k\in[0,\delta n]}
\binom{\lfloor  n/\alpha\rfloor}{k}  (2\Delta)^k
\sum_{\substack{j_1,\dots,j_k\geq 1:\\j_1+\dots+j_k=n}}
\siB_{j_1-1} \cdots \siB_{j_k-1}\\
&\qq +
2 \lfloor  n/\alpha
\rfloor \sum_{j\in [2\alpha\delta n,n]} \siB_j \sigma_{n-j},
\nonumber
\end{align}
where $\Delta$ denotes the maximum degree of $G$.
Here the first factor of
$2$ reflects the two cases $|S_+|\leq \delta n$ and $|S_-|\leq \delta n$, the
integer $k$ is $|S_+|$ or $|S_-|$, the binomial coefficient gives the number
of choices for $S_+$ or $S_-$ as a subset of the quasi-geodesic, and the
factor $(2 \Delta)^k$ accounts for the choices of directions of segments
along the geodesic and of the omitted edges $(w_{b-1},w_b)$.  In the second
term, the factor $2\lfloor n/\alpha\rfloor$ bounds the possible choices of
the vertex $w_m$, and $j=n-m$.

Inequality \eqref{bound} implies that $\mu \le \muB$, as required. To check this,
assume on the contrary that $\muB < \mu$.
For any $\eps>0$, there exists $C=C(\eps)>0$
such that
$$
\siB_n\leq C (\muB+\epsilon)^n, \qq \si_n \le C(\mu+\eps)^n .
$$
Substituting this into \eqref{bound}, we obtain that
\begin{align*}
\si_n(v_0)
&\le C'n \binom{ n/\alpha}{\de n}\binom n{\de n} \bigl[(\muB+\eps)(2\Delta)^\delta\bigr]^n \\
&\qq +C''n \bigl[(\muB+\eps)^{2\alpha\de}(\mu+\eps)^{1-2\alpha\de}\bigr]^{n}
\end{align*}
where $C',C''$ may depend on $\alpha$ and $\epsilon$,
and the integer-part symbols in the binomial coefficients have been suppressed to simplify the notation.
Therefore, by \eqref{connconst},
$$
\mu \le \max\left\{(\muB+\eps) f(\de),(\muB+\eps)^{2\alpha\de}(\mu+\eps)^{1-2\alpha\de} \right\},
$$
for some $f(\de)$ satisfying $f(\de) \downarrow 1$ as $\de \downarrow 0$.
Let $\eps\downarrow 0$.
Since $\mu/\muB>1$ by assumption, this is a contradiction for small positive $\de$.
\end{proof}

\section*{Acknowledgements}
This work was supported in part by the Engineering and Physical Sciences
Research Council under grant EP/103372X/1. It was begun during a visit by GRG
to the Theory Group at Microsoft Research, where the question of forward
extendability was posed by the audience at a lunchtime seminar devoted to the
results of \cite{GrL1}.  We thank Russell Lyons for many valuable
discussions, and Alan Hammond and Gordon Slade for pointing out relevant
references.

\bibliography{ext}
\bibliographystyle{amsplain}

\end{document}